\documentclass{article}
\usepackage[utf8]{inputenc}

\usepackage{enumerate}

\usepackage{amsmath, amsthm, amssymb, amsfonts, mathtools, mathrsfs, dsfont}
\usepackage{hyperref}
\usepackage{soul}

\usepackage{graphicx}
\usepackage{tikz}

\theoremstyle{plain}
\newtheorem{lemma}{Lemma}
\newtheorem{theorem}[lemma]{Theorem}
\newtheorem{proposition}[lemma]{Proposition}

\newtheorem{corollary}[lemma]{Corollary}
\newtheorem{remark}[lemma]{Remark}
\newtheorem*{defi}{Definition}
\newtheorem{defin}[lemma]{Definition}
\newtheorem{question}[lemma]{Question}

\newcommand{\ler}[1]{\left( #1 \right)} 

\newcommand{\Hom}{\mathsf{hom}}
\newcommand{\inj}{\mathsf{inj}}
\newcommand{\surj}{\mathsf{surj}}
\newcommand{\ihat}{\hat{i}}
\newcommand{\irhat}{\hat{i}_R}
\newcommand{\hhat}{\hat{h}}
\newcommand{\hrhat}{\hat{h}_R}
\newcommand{\Gq}{(G_q)_q^\infty}
\newcommand{\Gpp}{(G_q)_{q\in \text{Q}}}

\newcommand{\Rq}{(R_q)_q^\infty}

\newcommand{\prob}{\mathbb{P}}

\title{Logarithmic convergence of finite projective planes}
\author{Márton Borbényi\thanks{Eötvös Loránd University and HUN-REN Alfréd Rényi Institute of Mathematics, Budapest, Hungary. Email: \texttt{borbenyi.marton@renyi.hu}}, 
Panna Tímea Fekete \thanks{HUN-REN Alfréd Rényi Institute of Mathematics and Department of Computer Science and Information Theory, Faculty of Electrical Engineering and Informatics, Budapest University of Technology and Economics, Budapest, Hungary. Email: \texttt{fekete.panna.timea@renyi.hu}},
Aranka Hrušková\thanks{Weizmann Institute of Science, Israel. Email: \texttt{umim.cist@gmail.com}},
Ander Lamaison\thanks{Universidad P\'ublica de Navarra, Pamplona, Spain. Email: \texttt{ander.lamaison@unavarra.es}}}
\date{June 2026}

\begin{document}

\maketitle

\begin{abstract}
    In this paper, we study the so-called log-convergence of graphs defined by Balázs Szegedy \cite{Szegedy2015SparseGL}. We answer positively his Question 4 whether the sequence of the incidence graphs of projective planes over finite fields log-converges and whether the limit coincides with that of a particular random graph model.
\end{abstract}

\section{Introduction}

With roots in the Aldous--Hoover theory of exchangeable arrays, modern graph-limit theory first crystallised in the setting of dense graph sequences through the theory of graphons~\cite{AldousExchangeable, exchangeableVSgraphlimits, LovSzegedy2006, BorgsChayesLovSosVesztergombiI}, and has since been extended to other discrete structures, including hypergraphs, permutations, posets, and matroids~\cite{ElekSzegedy2012, permutons, posetsJanson, matroids}.
Meanwhile, other regimes of asymptotic edge density, such as bounded-degree graphs, have their own limit theories~\cite{BenjaminiSchramm2001, LeftRightBounded, HatLovSzegLocalGlobal}.
By contrast, sparse graph limits admit several inequivalent approaches, including normalised cut-metric theories, $L^p$ graphons, graphon-process and graphex theories, and action convergence~\cite{BollRiordanSparse, LpGraphonsI, GraphonProcesses, ActionConvergence}.

One such approach, called logarithmic convergence, was proposed in 2015 by Bal\'azs Szegedy~\cite{Szegedy2015SparseGL}. This notion was motivated by Sidorenko's conjecture \cite{Sidorenko} in which the central inequality becomes linear after taking logarithms. For technical reasons, we will restrict ourselves to considering only bipartite graphs, although Szegedy's definition applies to general graphs.

We denote by $\mathcal{B}$ the set of finite bipartite graphs whose vertex classes are labelled $V_1$ and $V_2$. That is,
\[
\mathcal{B}=\left\{G=(V_1(G),V_2(G),E(G)) : E(G)\subseteq V_1(G)\times V_2(G)\right\}.
\]
A homomorphism between two elements of $\mathcal{B}$ is a graph homomorphism that is label-preserving. That is, we will only count maps $\varphi:H\to G$ which are graph homomorphisms in the classical sense \emph{and} for which $\varphi(V_1(H))\subseteq V_1(G)$ and $\varphi(V_2(H))\subseteq V_2(G)$. Subsequently, this gives rise to the definition of the bipartite density of $H\in\mathcal{B}$ in $G\in\mathcal{B}$ as
\[
t_B(H,G):=\frac{\Hom(H,G)}{|V_1(G)|^{|V_1(H)|}|V_2(G)|^{|V_2(H)|}},
\]
where $\Hom(H,G)$ is the number of homomorphisms from $H$ to $G$.

The quantity $t_B(H,G)$ can be interpreted also as the probability that a randomly chosen label-preserving map $V(H)\to V(G)$ is a homomorphism, in which view
\[
d(H,G):=-\log t_B(H,G)
\]
denotes the Kullback-Leibler divergence of the uniform distribution on homomorphisms with respect to the uniform measure on all label-preserving maps.

\begin{defi}[Log-convergence]
Let $\mathcal{B}_0$ be the class of graphs of $\mathcal{B}$ with at least one edge. We say that a sequence $(G_n)_n^\infty$ of graphs in $\mathcal{B}_0$ is log-convergent if, for every pair of graphs $H_1, H_2\in \mathcal{B}_0$, the sequence
\begin{equation}\label{eqn:log_ratios}
\left(\frac{-\log t_B(H_1,G_n)}{-\log t_B(H_2,G_n)}\right)_n
\end{equation}
of real numbers has a limit.
\end{defi}

This definition can, however, be greatly simplified. Let $h(H,G)$ be the particular ratio
\[h(H,G)=\frac{d(H,G)}{d(K_2,G)}=\frac{-\log t_B(H,G)}{-\log t_B(K_2,G)},\]
where $K_2\in\mathcal{B}_0$ is the edge with partition. Then the following lemma tells us that, in a sense, log-convergence is in fact only concerned with the quantities $\log t_B(H,G)$ modulo edge density.

\begin{lemma}[Lemma 4.2 in \cite{Szegedy2015SparseGL}]\label{Lemma:logconvergent_h}
A graph sequence $(G_n)_n^\infty$ in $\mathcal{B}_0$ is log-convergent if and only if
\[\lim\limits_{n\to\infty}h(H,G_n) = \lim\limits_{n\to\infty}\frac{-\log t_B(H,G_n)}{-\log t_B(K_2,G_n)}\]
exists for every $H\in\mathcal{B}_0$. Every graph sequence in $\mathcal{B}_0$ has a log-convergent subsequence.
\end{lemma}

The ratios of the form (\ref{eqn:log_ratios}) are not defined when $G_n$ is a complete bipartite graph. In this case, following the same criterion as Szegedy, we define $h(H,G)=|E(H)|$ for any complete bipartite graph $G$.

Lemma \ref{Lemma:logconvergent_h} tells us that the trivial limit object for a log-convergent sequence $(G_n)_{n=1}^\infty$ is the vector $\left(\lim_n h(H_1,G_n),\lim_n h(H_2,G_n),\dots\right)\in\mathbb{R}_{\geq0}^{\mathcal{B}_0}$, where $H_1,H_2,\dots$ is some fixed enumeration of $\mathcal{B}_0$. As of now, the theory of log-convergence lacks a good analytical or algebraic limit object akin to graphon or graphing, and so whenever we will be talking about a limit, we will simply mean an element of $\mathbb{R}_{\geq0}^{\mathcal{B}_0}$.

Hand in hand with any theory of graph convergence goes a random model reflecting its salient features.
When analytical limit objects like graphons are at our disposal, we can typically use them to generate random graph sequences that almost surely converge to the limit object we started with. While as pointed out above, we currently do not have such a non-trivial limit, Szegedy introduced a random bipartite model that satisfies the almost sure convergence and captures a wide array of structural scenarios.
As opposed to the Erdős-Rényi model $\mathcal{G}(n,p)$ which has only one parameter determining the asymptotic behaviour, namely the edge density $p$, Szegedy's model requires two parameters: $\beta$ accounts for edge density, and $\alpha$ accounts for the relative sizes of the two vertex classes in the bipartition. This random model $R(n,\beta,\alpha)$ will be defined in Section~\ref{section:preliminaries}.
Szegedy proved that, for any fixed choice of the two parameters, the sequence $\left(R(n,\beta,\alpha)\right)_{n=1}^\infty$ of random bipartite graphs log-converges with probability 1, and gave an explicit description of the limit $R(\beta, \alpha)\in\mathbb{R}_{\geq0}^{\mathcal{B}_0}$.

While studying the random model $R(n,\beta,\alpha)$, Szegedy asked whether the incidence graphs of finite projective planes are pseudorandom in the log-convergence framework. In particular, let $PG(2,q)$ be the projective plane over the finite field $\mathbb{F}_q$. Then its incidence graph $G_q$
is the bipartite graph in which the class $V_1$ is the set of points of $PG(2,q)$, the class $V_2$ is the set of lines of $PG(2,q)$, and an edge is drawn between a point $p$ and a line $\ell$ if $p\in\ell$. Szegedy asked whether the sequence of incidence graphs of projective planes is log-convergent, when the parameter $q$ ranges over the set of primes. He further asked, in case the sequence is convergent, whether the limit is $R(3/4, 1/2)$, the same as the limit of a specific sequence of random bipartite graphs. We answer both questions in the affirmative, even when $q$ is allowed to be a prime power.

\begin{theorem}\label{thm:main}
The sequence $\Gpp$ of the incidence graphs of the projective planes $PG(2,q)$, where $q$ ranges over the set $Q$ of prime powers, is log-convergent. Moreover, its limit is $R(3/4, 1/2)$.
\end{theorem}
This statement is remarkable, because if we change the definition of the density $t_B$ to count only injective homomorphisms rather than all homomorphisms, the incidence graph of the projective plane behaves differently than the random bipartite graph. We will see some examples of this in the next section.

\addvspace{20pt}

\textbf{Acknowledgements.}
This project was born at the workshop \textit{Interfaces of the Theory of Combinatorial Limits} held in Erdős Center in Budapest in March 2022.

\section{Preliminaries}\label{section:preliminaries}

In this paper, all graphs mentioned belong to $\mathcal{B}_0$, and in particular are finite and bipartite. Through the rest of the paper, $\inj(H,G)$ and $\surj(H,G)$ denote, respectively, the number of injective homomorphisms and surjective homomorphisms from $H$ to $G$ (remember that we only count those homomorphisms sending $V_i(H)$ to $V_i(G)$).

Earlier, we defined log-convergence in terms of the fraction $\frac{-\log t_B(H_1,G_n)}{-\log t_B(H_2,G_n)}$.
Lemma \ref{Lemma:logconvergent_h} says that the only parameters that we are interested in are the \emph{exponents} to which we need to raise the edge density in order to obtain the other densities $t_B(H,G)$. This is what allows us to distinguish between graph sequences which, in the classical definition of graph convergence, would have as limit the zero graphon. This control over the exponents of densities is reflected in the following definition of a random bipartite graph model.
\begin{defi}[Random graph model]
Let $\beta\in(0,1]$ and $\alpha\in(0,1)$ be fixed. We denote by $G(n,\beta,\alpha)$ the probability distribution on bipartite graphs with $|V_1|=\lceil n^{\alpha} \rceil$, $|V_2|=\lceil n^{1-\alpha} \rceil$ given by including each of the $\lceil n^{\alpha}\rceil\lceil n^{1-\alpha}\rceil$ possible edges with probability $n^{\beta-1}$, independently of each other.
\end{defi}

Szegedy proves in \cite{Szegedy2015SparseGL} that for every fixed $\beta, \alpha>0$, a sequence drawn from $G(n,\beta,\alpha)$ log-converges with probability 1, and denotes by $R(\beta,\alpha)$ the collection of the limits $R(\beta,\alpha,H)=\lim_nh\left(H,G(n,\beta,\alpha)\right)$. The quantity $R(\beta,\alpha,H)$ is in fact defined to be
\[
\min_{H'\in\mathcal{C}(H)}\left\{|E(H')|+\frac{\alpha}{1-\beta}\ler{|V_1(H)-|V_1(H')|}+\frac{1-\alpha}{1-\beta}\ler{|V_2(H)-|V_2(H')|}\right\}
\]
which in turn is proven to be equal to the limit in probability $\lim_nh\left(H,G(n,\beta,\alpha)\right)$.

\begin{remark}
    The question whether the limits $R(\beta,\alpha)$ for different values of the parameters $\beta\in(0,1]$ and $\alpha\in(0,1)$ are distinct has a somewhat surprising answer, depicted in Figure~\ref{fig:parameter_square}.
    When $\alpha>\beta$, then $R(\beta,\alpha,K_{1,2})=1+\frac{1-\alpha}{1-\beta}<2$, and when $\alpha+\beta<1$, then $R(\beta,\alpha,K_{2,1})=1+\frac{\alpha}{1-\beta}<2$, telling us that for any $(\alpha,\beta)$ in the lower triangle of the parameter space as illustrated in Figure~\ref{fig:parameter_square}, there is no other pair of parameters representing the same limit. We conclude the same for the upper triangle by running through $K_{k,\ell}$ for arbitrarily large integers $k,\ell$ as opposed to just $K_{1,2}, K_{2,1}$.
    However, in the left triangle ($\alpha\leq\beta$ and $\alpha+\beta\leq1$), the dominant collapse (cf. Definition~\ref{def:collapse}) of a connected test graph $H$ is always the star which merges all vertices of $V_1$ into a single one while keeping those of $V_2$ separate, while in the right triangle ($\alpha\geq\beta$ and $\alpha+\beta\geq1$), it is always the star which merges all vertices of $V_2$ into a single one, which allows us to only ever learn $\frac{\alpha}{1-\beta}$ in the former case and $\frac{1-\alpha}{1-\beta}$ in the latter.
\end{remark}

\begin{figure}
    \centering
    {
\begin{tikzpicture}[scale=7, font=\small]

  \fill[blue!10]   (0,0) -- (0,1) -- (0.5,0.5) -- cycle;
  \fill[red!10]    (1,0) -- (1,1) -- (0.5,0.5) -- cycle;
  \fill[green!10]  (0,0) -- (1,0) -- (0.5,0.5) -- cycle;
  \fill[yellow!20] (0,1) -- (1,1) -- (0.5,0.5) -- cycle;

  \draw[thick] (0,0) rectangle (1,1);

  \draw[thick, dashed] (0,0) -- (1,1);   
  \draw[thick, dashed] (0,1) -- (1,0);   

  \draw[->] (-0.08,0) -- (1.12,0) node[right] {$\alpha$};
  \draw[->] (0,-0.08) -- (0,1.12) node[above] {$\beta$};

  \foreach \x/\xl in {0.5/$\tfrac{1}{2}$, 1/$1$} {
    \draw (\x, 0.015) -- (\x, -0.015) node[below] {\xl};
  }
  \foreach \y/\yl in {0.5/$\tfrac{1}{2}$, 1/$1$} {
    \draw (0.015, \y) -- (-0.015, \y) node[left] {\yl};
  }
  \node[below left] at (0,0) {$0$};

  \node[rotate=45,  above left,  font=\footnotesize] at (0.72,0.72)
    {$\alpha=\beta$};
  \node[rotate=-45, above right, font=\footnotesize] at (0.28,0.72)
    {$\alpha+\beta=1$};


  \node[align=center, font=\small] at (0.5, 0.78)
    {distinct};

  \node[align=center, font=\small] at (0.5, 0.22)
    {distinct};

  \node[align=center, font=\scriptsize, text width=2.8cm] at (0.18, 0.5)
    {$R(\beta,\alpha)=R\ler{\beta'\!,\alpha'}$\\[2pt]
     if and only if\\$\dfrac{\alpha}{1-\beta}=\dfrac{\alpha'}{1-\beta'}$};

  \node[align=center, font=\scriptsize, text width=2.8cm] at (0.82, 0.5)
    {$R(\beta,\alpha)=R\ler{\beta'\!,\alpha'}$\\[2pt]
     if and only if\\$\dfrac{1-\alpha}{1-\beta}=\dfrac{1-\alpha'}{1-\beta'}$};

\end{tikzpicture}   
    }
    \caption{Depending on the values of the parameters $\beta\in(0,1]$ and $\alpha\in(0,1)$, the limits $R(\beta,\alpha)$ are not always necessarily distinct}
    \label{fig:parameter_square}
\end{figure}
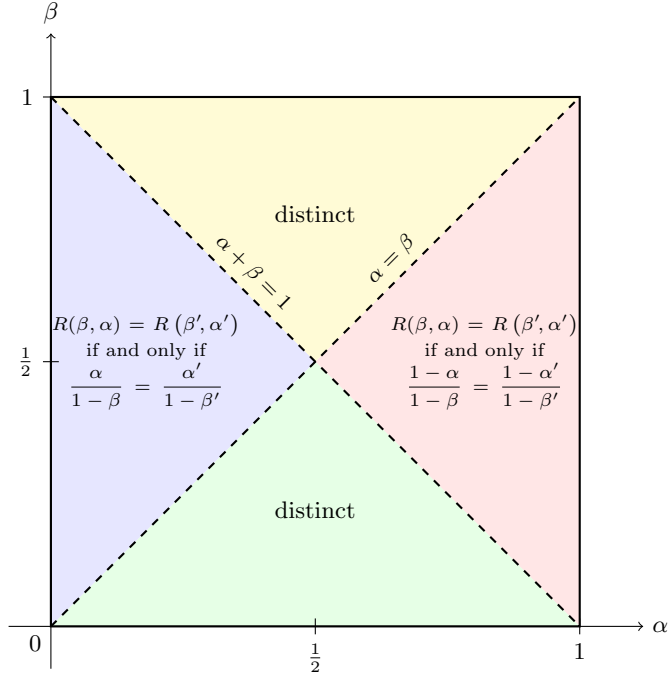

We now shift our attention to the graph sequence whose log-convergence we wish to prove.
For every finite projective plane, there is a positive integer $n$, known as the order of the plane, such that the plane has $n^2+n+1$ points, $n^2+n+1$ lines, $n+1$ points on each line, and $n+1$ lines through each point. It is a long-standing open question whether $n$ can ever be a number that is not a prime power. We denote by $P$ the set of primes and by $Q$ the set of prime powers.

For any division ring $K$, we can construct a projective plane $PG(2,K)$ by treating the 1-dimensional subspaces of $K^3$ as points and the 2-dimensional ones as lines.
Any finite division ring is necessarily a field, and for any prime power $q$, there exists a unique field with $q$ elements, allowing us to use the notation $PG(2,q)$ for the unique projective plane arising from this field.

Looking at the terms of the largest order, we can see that the graph $G(q^4, 3/4, 1/2)$ would have similar part sizes ($q^2$ each) and expected edge density ($\mathbb{E}\left[t_B(K_2,G)\right]=q^{-1}$) as $G_q$, the incidence graph of $PG(2,q)$. Szegedy asks whether the sequence $(G_q)_{q\in P}$ is pseudorandom, meaning, whether it converges to $R(3/4,1/2)$. In the rest of the paper, we will prove that $(G_q)_{q\in Q}$ converges to $R(3/4,1/2)$ by showing increasing similarity between $G_q$ and $R_q:=G(q^4,3/4,1/2)$.

The following definition plays a crucial role in the description of the limit in our main result, as well as in its proof.

\begin{defin}[Collapse]\label{def:collapse}
A graph $H'$ is a \emph{collapse}\footnote{Szegedy uses the term \emph{homomorphic image} for collapses in \cite{Szegedy2015SparseGL}. } of a graph $H$ if there is a graph homomorphism $\varphi$ from $H$ to $H'$ which is both vertex- and edge-surjective. Let $\mathcal{C}(H)$ denote the set of collapses of $H$.
\end{defin}

Since we only care about the ratio of logarithms, we are free to choose the base of such logarithm. Given the parametrization of our graphs, it seems natural to take $q$ as the base of the logarithm, because the edge density of our graph is $(1+o(1))q^{-1}$. That means that $\lim_{q\rightarrow\infty}-\log_qt_B(K_2, G_q)=1$, cancelling the denominator in Lemma~\ref{Lemma:logconvergent_h}. In addition, observe that $\log_qt_B(H, G_q)=\log_q\Hom(H, G_q)-|V(H)|\log_q(q^2+q+1)$, and so by Lemma~\ref{Lemma:logconvergent_h}, to guarantee the log-convergence of $(G_q)_{q\in Q}$, it is enough to verify that $\log_q\Hom(H, G_q)$ has a limit for every $H\in\mathcal{B}_0$.

First we are going to show that by restricting ourselves to an adequate subsequence of prime powers, we can assume that $\lim_{q\rightarrow \infty}\log_q \inj(H,G_q)$ and $\lim_{q\rightarrow\infty}\log_q\Hom(H, G_q)$ exist for all graphs $H$, where $\log_q0=-\infty$.

\begin{lemma}\label{lemma:subsequence}
There exists a sequence $(q_i)_{i=1}^\infty$ of prime powers such that $\ler{\log_{q_i}\inj(H, G_{q_i})}_{i=1}^\infty$ and $\ler{\log_{q_i}\Hom(H, G_{q_i})}_{i=1}^\infty$ converge for every $H\in\mathcal{B}_0$.
Moreover, if Theorem~\ref{thm:main} is false, then such a sequence can be chosen in a way which ensures that $\lim_{i\rightarrow\infty}\log_{q_i}\Hom(H, G_{q_i})\neq \lim_{i\rightarrow\infty}\log_{q_i}\Hom(H, R_{q_i})$ for at least one $H\in\mathcal{B}_0$.
\end{lemma}

\begin{proof}
For every $H\in\mathcal{B}_0$ and $q\in Q$, we have that $\inj(H, G_q)$ and $\Hom(H, G_q)$ are integers in the interval $\left[0, (q^2+q+1)^{|V(H)|}\right]$.
Since $q^2+q+1<q^3$ for $q\geq 2$, it follows that $\log_q\inj(H, G_q)$ and $\log_q\Hom(H, G_q)$ both lie in $\{-\infty\}\cup[0, 3|V(H)|]$. We therefore conclude that the sequence $\ler{\ler{\log_q\inj(H, G_q),\log_q\Hom(H, G_q)}}_{q\in Q}$ is a sequence in the space
\[X=\prod_{H\in\mathcal{B}_0} \left(\{-\infty\}\cup[0,3v(H)]\right)^2,\]
which is sequentially compact, thanks to being the product of countably many sequentially compact spaces.
Hence, there exists a subsequence of $\ler{\ler{\log_q\inj(H, G_q),\log_q\Hom(H, G_q)}}_{q\in Q}$ in which every coordinate converges.

If Theorem~\ref{thm:main} is false, then for some $H\in\mathcal{B}_0$, the sequence $\left(\log_q\Hom(H, G_q)\right)_{q\in Q}$ does not converge to $\lim_{q\rightarrow\infty}\log_q\Hom(H, R_q)$. This means that $\left(\log_q\Hom(H, G_q)\right)_{q\in Q}$ must have a different accumulation point. We select a sequence $\ler{q_i}_{i=1}^\infty$ of prime powers such that $\left(\log_{q_i}\Hom(H, G_{q_i})\right)_{i=1}^\infty$ tends to this accumulation point, then use the sequential compactness of $X$ to find a subsequence of $\ler{q_i}_{i=1}^\infty$ for which the corresponding points in $X$ converge.
\end{proof}

From this point on, we fix a subsequence $(q_i)_{i=1}^\infty$ of prime powers guaranteed by Lemma~\ref{lemma:subsequence}, and always assume only working within $\ler{G_{q_i}}_{i=1}^\infty$, even if we drop the index $i$ from $q_i$. This restriction is necessary: for example, $G_2$, which is the incidence graph of the Fano plane, is a subgraph of $G_q$ if and only if $q$ is a power of 2, meaning that the sequence $\left(\log_q\inj(G_2, G_q)\right)_{q\in Q}$ does not converge.

We are going to use the following notation:
\begin{align*}
\ihat(H)&:=\lim\limits_{q\rightarrow \infty}\log_q\inj(H,G_q),\\
\hhat(H)&:=\lim\limits_{q\rightarrow \infty}\log_q\Hom(H,G_q),\\
\irhat(H)&:=2|V(H)|-|E(H)|,\\
\hrhat(H)&:=\max_{H'\in\mathcal{C}(H)}\irhat(H').
\end{align*}

Given this notation, the reader might suspect to find an analogy between $\ihat$ and $\irhat$, and between $\hhat$ and $\hrhat$. Indeed, this relation is found upon considering the random graph $R_q$. A simple first-moment calculation reveals that
\[
\lim_{q\rightarrow\infty}\log_q\mathbb{E}\left[\inj(H,R_q)\right]=2|V(H)|-|E(H)|=\irhat(H).
\]
From there, classifying the homomorphisms from $H$ to $R_q$ by their image produces
\[\lim_{q\rightarrow\infty}\log_q\mathbb{E}\left[\Hom(H,R_q)\right]
=\lim_{q\rightarrow\infty}\log_q\sum_{H'\in\mathcal{C}(H)}\surj(H,H')\mathbb{E}\left[\inj(H',R_q)\right]
=\max_{H'\in\mathcal{C}(H)}\irhat(H')=\hrhat(H).\]

Szegedy~(\cite{Szegedy2015SparseGL}, see Theorem 3) proved a concentration result for the number of homomorphisms in the random bipartite graph $R(n, \beta, \alpha)$. His result implies that with probability 1, for every graph $H$, the number $\log_q\Hom(H,R_q)$ converges to $\hrhat(H)$.

The next proposition states that it is possible to compute the $\hhat$ of a graph from the $\ihat$ of its collapses, in the same way in which one can compute $\hrhat$ from $\irhat$.

\begin{proposition}\label{prop:max}
For every graph $H\in\mathcal{B}_0$, we have \[\hhat(H)=\max_{H'\in\mathcal{C}(H)}\ihat(H').\]
\end{proposition}

\begin{proof}
Let $\phi$ be a homomorphism from $H$ to $G_q$. Its image is a collapse $H'$ of $H$. In fact, $\phi$ can be expressed in a unique way as the composition of a surjective homomorphism from $H$ to $H'$ and an injective homomorphism from $H'$ to $G_q$. Recall that $\surj(H,H')$ denotes the number of label-preserving surjective homomorphisms from $H$ to $H'$. Then
\[\Hom(H,G_q)=\sum_{H'\in \mathcal{C}(H)}\surj(H,H')\inj(H',G_q).\]

Note that
\[
\max_{H'\in\mathcal{C}(H)}\ihat(H')
=\max_{H'\in\mathcal{C}(H)}\lim_{q\to\infty}\log_q\inj(H',G_q)
=\lim_{q\to\infty}\log_q\max_{H'\in\mathcal{C}(H)}\inj(H',G_q)
\]
by finiteness of $\mathcal{C}(H)$,
and thus
\begin{equation*}
    \begin{split}
        \log_q\Hom(H,G_q) & =\log_q\left(\sum_{H'\in \mathcal{C}(H)}\surj(H,H')\inj(H',G_q)\right) \\
        & \leq \log_q\left(\left(\sum_{H'\in \mathcal{C}(H)}\surj(H,H')\right) \max\limits_{H'\in \mathcal{C}(H)}\inj(H',G_q)\right) \\
        & = \log_q\left(\sum_{H'\in \mathcal{C}(H)}\surj(H,H')\right)+ \log_q\max_{H'\in\mathcal{C}(H)}\inj(H',G_q)\\
        & \rightarrow \max\limits_{H'\in \mathcal{C}(H)}\ihat(H') \quad \text{ as } q \rightarrow \infty,
    \end{split}
\end{equation*}
and 
\begin{equation*}
    \begin{split}
        \log_q\Hom(H,G_q) & =\log_q\left(\sum_{H'\in \mathcal{C}(H)}\surj(H,H')\inj(H',G_q)\right) \\
        & \geq \log_q \max\limits_{H'\in \mathcal{C}(H)}\inj(H',G_q) \\
        & \rightarrow \max\limits_{H'\in \mathcal{C}(H)}\ihat(H') \quad \text{ as } q \rightarrow \infty,
    \end{split}
\end{equation*}
giving us the desired identity.
\end{proof}
Throughout the paper we will use the following notations for graph operations.
\begin{defi}[Graph operations]
Let \(G = (V_1(G), V_2(G), E(G)) \in \mathcal{B}\) be a labelled bipartite graph. For any \(W\subseteq V_1(G) \cup V_2(G)\), let \(G-W\) denote the bipartite graph that we obtain by deleting the vertices of \(W\) and all edges incident to these vertices.

The \emph{wedge sum} of two graphs $H_1$ and $H_2$, at vertices $v_1\in V_i(H_1)$ and $v_2\in V_i(H_2)$, is obtained by taking disjoint copies of $H_1$ and $H_2$, and identifying the vertices $v_1$ and $v_2$. We denote a wedge sum by $H_1\vee H_2$.

The disjoint union of $H_1$ and $H_2$ connected by an edge at vertices $v_1\in V_i(H_1)$ and $v_2\in V_{3-i}(H_2)$ is, as its name indicates, the disjoint union of $H_1$ and $H_2$ together with an edge between $v_1$ and $v_2$.
\end{defi}

\begin{lemma}\label{lemma:homtransitivity}
Let $k$ be a non-negative integer and \(H\in \mathcal{B}_0\). Let us set
\[\inj_k(H,G_q):=\min_{v_1,...,v_k\in V(G_q)} \inj(H,G_q-\{v_1,...,v_k\}).\]
Then for every $H\in\mathcal{B}_0$ and any fixed non-negative $k$,
\[\ihat(H) =\lim_{q\rightarrow\infty}\log_q\inj_k(H,G_q).\]
\end{lemma}

\begin{proof}
We can assume that the left-hand side is not $-\infty$, because otherwise the lemma is trivial.

Let $w\in V(H)$, $v\in V(G_q)$. Then the probability that a uniformly randomly chosen injection $f$ takes $w$ to $v$ is $0$ or $1/(q^2+q+1)$ by transitivity of the projective planes $PG(2,q)$. Thus by the union bound, for any $A=\{v_1,...,v_k\}\subset V(G_q)$
\[
\prob(f(H)\cap A\not=\emptyset)\le \frac{|V(H)|k}{q^2+q+1}=O(q^{-2}).
\]
Hence
\begin{equation*}
    \begin{split}
        \log_q\inj_k(H,G_q)- \log_q\inj(H,G_q)=\log_q\frac{\inj_k(H,G_q)}{\inj(H,G_q)}=\log_q\left[\min_{|A|=k}\prob(f(H)\cap A=\emptyset)\right]\to 0
    \end{split}
\end{equation*}
as \(q\) tends to infinity, so we are done.
\end{proof}

\begin{lemma}\label{lemma:unions}
Let $H_1$ and $H_2$ be two labelled bipartite graphs. Let $H$ be obtained from $H_1$ and $H_2$ by taking either their disjoint union ($H_1\dot\cup H_2$), wedge sum ($H_1\vee H_2$) or the disjoint union connected by an edge. Then $\ihat(H)$ and $\irhat(H)$
are as in Table~\ref{table:tableofunions}, independently of which two vertices are identified by the wedge or connected by the edge.
\renewcommand{\arraystretch}{2}
\begin{table}[!ht]
\begin{center}
\begin{tabular}{|c|c|c|c|}
\hline
\multicolumn{1}{|l|}{} & \multicolumn{1}{c|}{$H=H_1\dot\cup H_2$} & \multicolumn{1}{c|}{$H=H_1\vee H_2$}                                         & \multicolumn{1}{c|}{disjoint union connected by an edge}                                    \\ \hline
$\ihat(H)$ & $\ihat(H_1)+\ihat(H_2)$ & $\ihat(H_1)+ \ihat(H_2)-2$ & $\ihat(H_1)+ \ihat(H_2)-1$  \\ \hline
$\irhat(H)$ & $\irhat(H_1)+\irhat(H_2)$ & $\irhat(H_1)+ \irhat(H_2)-2$ & $\irhat(H_1)+ \irhat(H_2)-1$ \\ \hline
\end{tabular}
\end{center}
\caption{Results of Lemma~\ref{lemma:unions}}
\label{table:tableofunions}
\end{table}
\renewcommand{\arraystretch}{1}
\end{lemma}

\begin{proof}
For $\irhat$, the result is an easy computation. 

For $\ihat$, the first column comes from Lemma \ref{lemma:homtransitivity}, because
\[
\inj(H_1, G_q)\inj_{|V(H_1)|}(H_2, G_q)\leq \inj(H_1\dot\cup H_2, G_q)\leq \inj(H_1, G_q)\inj(H_2, G_q).
\]

For the second column, we have
$\inj(H_1\vee H_2, G_q)\leq\frac{\inj(H_1, G_q)\inj(H_2, G_q)}{q^2+q+1}$ by transitivity of $G_q$. For the lower bound we need an argument similar to Lemma~\ref{lemma:homtransitivity}. 

We can assume that $\ihat(H_1), \ihat(H_2)\neq-\infty$, as otherwise the statement is trivial. Up to a change in the indices, we can suppose that $v_1\in V_1(H_1)$. Fix a copy of $H_1$ in $G_q$, and consider the point $p$ which is the image of $v_1$. Consider the subgroup of automorphisms of $PG(2,q)$ which fixes $p$. This group is transitive on the set of points distinct from $p$, transitive on the set of lines containing $p$, and transitive on the set of lines not containing $p$. Next consider a copy of $H_2$ in $G_q$ for which the image of $v_2$ is $p$. By transitivity of $G_q$, there are exactly $\frac{\inj(H_2, G_q)}{q^2+q+1}$ such copies. Now consider the image of $H_2$ under the automorphism group described above. Given a vertex $w_1\in V(H_1)\setminus\{v_1\}$ and a vertex $w_2\in V(H_2)\setminus\{v_2\}$, the proportion of homomorphisms that send $w_1$ to $w_2$ is at most $\frac{1}{q+1}$. This means that the proportion of automorphisms of the copy of $H_2$ that intersect non-trivially the fixed copy of $H_1$ is at most $O(q^{-1})$, and so $\inj(H_1\vee H_2, G_q)\geq(1-O(q^{-1}))\frac{\inj(H_1, G_q)\inj(H_2, G_q)}{q^2+q+1}$.

Finally, the third column can be obtained from the second, because the disjoint union of $H_1$ and $H_2$ joined by an edge can be written as $(H_1\vee K_2)\vee H_2$.
\end{proof}

\begin{corollary}\label{cor:remove}
Let $H$ be a labelled bipartite graph, and let $v\in V(H)$.
\begin{itemize}
    \item If $d(v)=0$, then $\ihat(H-v)=\ihat(H)-2$.
    \item If $d(v)=1$, then $\ihat(H-v)=\ihat(H)-1$.
    \item If $d(v)\geq2$, then $\ihat(H-v)\geq\ihat(H)$.
\end{itemize}
\end{corollary}

\begin{proof}
An isolated vertex has $\ihat(K_1)=2$. If $d(v)=0$, then $H$ is the disjoint union of $H-v$ and $K_1$. If $d(v)=1$, then $H$ is the union of $H-v$ and $K_1$ connected by an edge and hence $\ihat(H)=\ihat(H-v)+\ihat(K_1)-1$.

If $d(v)\geq 2$, then the fact that any pair of distinct vertices in $G_q$ has at most one common neighbour implies that any injective homomorphism from $H-v$ to $G_q$ extends in at most one way to $H$, so \[\ihat(H)=\lim\limits_{q\rightarrow\infty}\log_q\inj(H,G_q)\leq \lim\limits_{q\rightarrow\infty}\log_q\inj(H-v,G_q)=\ihat(H-v).\qedhere\]
\end{proof}

Finally, we will need the following lemma. It seems out of place right now, but it will be used in the final steps of the proof of Theorem \ref{thm:main}, in which 2-connectivity is important.

\begin{lemma}\label{lemma:removeedge}
   Let $G$ be a $2$-connected graph with $\delta(G)\geq3$. Then there is an edge $e\in E(G)$ such that the graph $G-\{e\}$ obtained by removing $e$ from $G$ is still $2$-connected.
\end{lemma}
\begin{proof}
   By the main theorem of \cite{chartrand1972critically}, there is a vertex $v\in V(G)$ such that $G-v$ is still 2-connected. Suppose $u_1,\dots,u_{\deg(v)}\in V(G)$ were the neighbours of $v$. Then adding a vertex $w$ to $G-v$ and connecting it to $u_2,\dots,u_{\deg(v)}$ is adding a vertex of degree at least 2 to a 2-connected graph, so the resulting graph $G-v+w$ is still 2-connected. But $G-v+w$ is isomorphic to $G-\{u_1v\}$, so $u_1v$ is an edge whose removal preserves 2-connectivity.
\end{proof}

Before starting with the proof of our main theorem, let us rephrase it using the notation that we have introduced. The following statement is, using Lemma~\ref{lemma:subsequence}, equivalent to Theorem~\ref{thm:main}.

\begin{theorem}\label{thm:quasirandom}
For every bipartite graph $H$ we have $\hhat(H)=\hrhat(H)$.
\end{theorem}

If instead of homomorphisms, we only count injective homomorphisms, we can see that the graphs $G_q$ and $R_q$ behave differently. An obvious example is $C_4$: in a projective plane, every pair of lines intersects in exactly one point, so $\ihat(C_4)=-\infty$, while we have $\irhat(C_4)=4$. On the other hand, there are examples in which $\ihat(H)>\irhat(H)$, i.e., the graph appears more often in $G_q$ than in $R_q$. Such examples include the incidence graphs of the point-line arrangements corresponding to Pappus's theorem (denoted by $\bar P$) and Desargues's theorem (denoted by $\bar D$). These graphs satisfy $\irhat(\bar P)=9$, $\ihat(\bar P)=10$, $\irhat(\bar D)=10$ and $\ihat(\bar D)=11$.

\begin{figure}[h]
\begin{centering}
\includegraphics[width=0.8\textwidth]{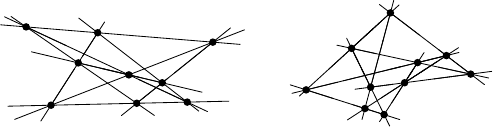}
\caption{Point-line arrangements corresponding to Pappus's theorem (left) and Desargues's theorem (right).}
\par\end{centering}
\end{figure}

The reason why these differences between $G_q$ and $R_q$ do not immediately disprove Theorem~\ref{thm:main} is that a significant proportion of the homomorphisms from the graphs described above to $G_q$ and $R_q$ are not injective. By Proposition~\ref{prop:max} and the definition of $\hrhat$, in order to determine the values of $\hhat$ and $\hrhat$ of a graph $H$, one needs to consider the values of $\ihat$ and $\irhat$ not just for $H$ itself, but also for its collapses. In these cases, there are collapses for which the values of $\ihat$ and $\irhat$ are greater than or equal to those of $H$. Indeed, consider the star $K_{1,t}$ which has $\ihat(K_{1,t})=\irhat(K_{1,t})=t+2$. Observe that $K_{1,2}$ is a collapse of $C_4$,  $K_{1,9}$ is a collapse of $\bar P$, and $K_{1,10}$ is a collapse of $\bar D$, each obtained by mapping all vertices of one side of the bipartition to the center of the star.

\section{Proof of Theorem~\ref{thm:quasirandom}}

We mentioned in the previous paragraph how, for some graphs $H$, most homomorphisms from $H$ to $G_q$ and $R_q$ are not injective. We will establish a definition for the opposite case, that is, when almost all homomorphisms are injective. This definition will play a crucial role in the proof of our main theorem.

\begin{defi}
A graph $H$ is \emph{$P$-critical} if $\ihat(H)>\ihat(H')$ for all proper collapses $H'$ of $H$. Analogously, $H$ is \emph{$R$-critical} if $\irhat(H)>\irhat(H')$ for all proper collapses $H'$ of $H$.
\end{defi}

\begin{lemma}\label{lemma:dominantcollapse}
    Every graph $H$ has a collapse $H'\in\mathcal{C}(H)$ such that $\ihat(H')\geq\ihat(H)$ and $H'$ is $P$-critical.
    Similarly, it has a collapse $H''$ such that $\irhat(H'')\geq\irhat(H)$ and $H''$ is $R$-critical.
\end{lemma}
\begin{proof}
    If $H$ itself is $P$-critical, we are done.
    Otherwise there is a proper collapse $H'$ of $H$ with $\ihat(H')\geq\ihat(H)$. Since the relation of being a collapse is transitive and we have $1\leq|V(H')|<|V(H)|$, we can continue the process with $H'$ and must stop after finitely many steps.
    The proof of the analogous statement for $R$-criticality is the same.
\end{proof}

The following proposition, which will be the main part of our proof, implies Theorem~\ref{thm:quasirandom}:

\begin{proposition}\label{prop:critical}
Suppose that a graph $H\in\mathcal{B}_0$ satisfies $\ihat(H)\neq\irhat(H)$. Then $H$ is not $P$-critical nor $R$-critical.
\end{proposition}

\begin{proof}[Proof of Theorem~\ref{thm:quasirandom}]

 Suppose that the theorem is not true. Let $H$ be a graph with $\hhat(H)\neq\hrhat(H)$. Suppose that $\hhat(H)>\hrhat(H)$ (the other case is analogous). By Proposition~\ref{prop:max}, there exists at least one collapse $H'$ of $H$ which satisfies $\ihat(H')=\hhat(H)$.
 Moreover, by Lemma \ref{lemma:dominantcollapse}, we can choose $H'$ to be $P$-critical. But then $\ihat(H')=\hhat(H)>\hrhat(H)\geq\irhat(H')$, contradicting Proposition \ref{prop:critical}.
\end{proof}

This is not a big improvement, because proving Proposition~\ref{prop:critical} is essentially as hard as proving Theorem~\ref{thm:quasirandom}. The main difference between the two is that Proposition~\ref{prop:critical} is written in a form that lends itself to be proved using induction. In general, it will be easier to relate the values of $\ihat(H)$ and $\irhat(H)$ to those of the subgraphs of $H$ than the values of $\hhat(H)$ and $\hrhat(H)$. 

\begin{proof}[Proof of Proposition~\ref{prop:critical}]
Suppose the statement is not true and let $H$ be a minimal counterexample (minimal with respect to number of vertices). We will do a thorough case analysis to show that $H$ cannot exist. 

We start by observing that $H$ must be connected. Indeed, if $H$ is not connected, then it is the disjoint union of two non-empty graphs $H_1$, $H_2$. Then by Lemma~\ref{lemma:unions}, we know that $\ihat(H_1)+\ihat(H_2)=\ihat(H)$ and $ \irhat(H_1)+\irhat(H_2)=\irhat(H)$. Thus we can assume without loss of generality that $\ihat(H_1)\not=\irhat(H_1)$. By minimality of $H$, we have that $H_1$ is not $P$-critical nor $R$-critical. That is, there exist proper collapses $H'_1$ and $H''_1$ of $H_1$ such that $\ihat(H'_1)\geq \ihat(H_1)$ and $\irhat(H''_1)\geq \irhat(H_1)$. But then again by Lemma~\ref{lemma:unions}, we obtain $\ihat(H'_1\dot\cup H_2)\geq \ihat(H)$ and $\irhat(H''_1\dot\cup H_2)\geq \irhat(H)$. Since $H'_1\dot\cup H_2$ and $H''_1\dot\cup H_2$ are proper collapses of $H$, we conclude that $H$ is not $R$-critical nor $P$-critical, proving that $H$ is not a counterexample in the first place.

Next we consider different cases depending on the minimum degree of $H$:

\textbf{Case 0: $\delta(H)=0$.} $H$ is in $\mathcal{B}_0$, and so it cannot be an isolated vertex. But $H$ must be connected, so we must conclude that $\delta(H)=0$ cannot happen.

\bigskip

\textbf{Case 1: $\delta(H)=1$.} Let $v$ be a vertex in $H$ of degree 1. Then Corollary~\ref{cor:remove} tells us that
\[
\ihat(H-v)=\ihat(H)-1.
\]
On the other hand, the definition of $\irhat$ says that
\[\irhat(H-v)=2|V(H-v)|-|E(H-v)|=2(|V(H)|-1)-(|E(H)-1)=2|V(H)|-|E(H)|-1=\irhat(H)-1.\]
Since $H$ is a counterexample, we must have $\ihat(H)\neq\irhat(H)$, and in particular by the two equations above,
\[
\ihat(H-v)=\ihat(H)-1\neq\irhat(H)-1=\irhat(H-v).
\]
But $H$ was a \emph{minimal} counterexample, so $H-v$ is not a counterexample, and hence it cannot be $P$-critical nor $R$-critical.
In particular, $H-v$ not being $P$-critical tells us that there is a proper collapse $(H-v)'$ of $H-v$ satisfying $\ihat(H-v)\leq\ihat((H-v)')$. Let $+v$ denote the operation of attaching the previously deleted vertex back to (the node representing) the original neighbour of $v$. With this notation, $H-v+v=H$, and $(H-v)'+v$ is a proper collapse of $H$. Then
\begin{align*}
    \ihat(H-v)&\leq\ihat((H-v)')\\
    \ihat(H)=\ihat(H-v)+1&\leq\ihat((H-v)')+1=\ihat((H-v)'+v)
\end{align*}
meaning that $H$ is not $P$-critical.

$H-v$ is also not $R$-critical, and so there is some proper collapse $(H-v)''$ of $H-v$ satisfying $\irhat(H-v)\leq\irhat((H-v)'')$, which implies that
\begin{align*}
\irhat(H)=\irhat(H-v)+1\leq\irhat((H-v)'')+1
&=2|V((H-v)'')|-|E((H-v)'')|+1\\
&=2(|V((H-v)''+v)|-1)-(|E((H-v)''+v)|-1)+1\\
&=2|V((H-v)''+v)|-|E((H-v)''+v)|
=\irhat\left((H-v)''+v\right).
\end{align*}
But this means that $H$ was actually not $R$-critical either, so in particular it is, after all, not a counterexample to the proposition.

\bigskip

\textbf{Case 2: $\delta(H)=2$.} Let $v$ be a vertex in $H$ of degree 2.

\hspace{1cm} \textbf{Case 2.1: $H$ is $P$-critical.} Then
\begin{equation}\label{eqn:inj_of_critical_H}
\hhat(H) = \ihat(H) \leq \ihat(H-v)
\end{equation}
by Proposition \ref{prop:max} and Corollary~\ref{cor:remove}.
Now we consider two cases depending on whether \(H-v\) is $P$-critical or not.

\hspace{2cm} \textbf{Case 2.1.1: \(H-v\) is not $P$-critical.} Then there exists a proper collapse \(H'\) of \(H-v\) such that \(\ihat(H') \geq \ihat(H-v)\).

By the definition of collapse, $H'$ is the image of $H-v$ after a homomorphism. Let $x$ and $y$ be the neighbours of $v$ in $H$, and let $x'$ and $y'$ be their images in $H'$ (potentially we have $x'=y'$). Construct a graph $H'+w$ by adding a new vertex $w$, and attaching it to $x'$ and $y'$. Then $H'+w$ is a proper collapse of $H$, obtained by extending the homomorphism $H-v\rightarrow H'$ to map $v$ to $w$.
Moreover, in every homomorphic copy of $H'$ in $G_q$, there is 1 or $q+1$ vertices where $w$ could be mapped, giving that $\Hom(H'+w,G_q)\geq\Hom(H',G_q)$ for every $q$, and subsequently $\hhat(H'+w)\geq\hhat(H')$.

Then
\begin{equation*}
    \begin{split}
        \max_{H''\in\mathcal{C}(H'+w)}\ihat(H'')=\hhat (H' + w) \geq \hhat (H') \geq \ihat(H') \geq \ihat(H-v) \geq \ihat (H),
    \end{split}
\end{equation*}
contradicting the criticality of $H$ in $\Gq$.

\hspace{2cm} \textbf{Case 2.1.2: \(H-v\) is $P$-critical.} Since \(H\) is a minimal counterexample to the proposition, we know that $H-v$ being $P$-critical must mean that
\[
\ihat (H-v) = \irhat (H-v).
\]

Also similarly as above, we have that $\hhat(H-v)\leq\hhat(H)$,
because there is always at least one choice for placing $v$ in $G_q$ after placing $H-v$.
After extending inequality (\ref{eqn:inj_of_critical_H}) with these expressions, we get
\[
\hhat(H-v)\leq\hhat(H)=\ihat(H)\leq\ihat(H-v)\leq\hhat(H-v),
\]
which therefore must have equality throughout.
But then
\[
\ihat(H)=\ihat(H-v)=\irhat(H-v)=2(|V(H)|-1)-(|E(H)|-2)=\irhat(H),
\]
and so $H$ could not be a counterexample to the proposition.

\hspace{1cm} \textbf{Case 2.2: $H$ is $R$-critical, but not $P$-critical.}
We consider two cases depending on whether \(\ihat(H-v) = \irhat (H-v)\) or not.

\hspace{2cm} \textbf{Case 2.2.1: \(\ihat(H-v) \neq \irhat(H-v)\).} By minimality of $H$, \(H-v\) is not $R$-critical. Thus there exists a proper collapse \(H'' \in \mathcal{C}(H-v)\) such that
\[\irhat(H'') \geq \irhat (H-v).\]
It is then possible to add a vertex \(w\) with \(\deg(w) \in\{ 1,2\}\), such that \(H'' + w \) is a proper collapse of \(H\), just as in Case 2.1.1. Then
\[
\irhat(H'' + w) \geq 2(|V(H'')|+1)-(|E(H'')|+2) = \irhat(H'') \geq \irhat(H-v) = \irhat(H),
\]
contradicting the criticality of \(H\) in $\Rq$.

\hspace{2cm} \textbf{Case 2.2.2: \(\ihat(H-v) = \irhat(H-v)\).}
Let $H'$ be a proper collapse of $H$ such that $\hhat(H)=\ihat(H')=\hhat(H')$ and $H'$ is $P$-critical. Such $H'$ exists because $H$ is not $P$-critical and by Lemma \ref{lemma:dominantcollapse}.
Then
\[
\ihat(H-v)=\irhat(H-v)=\irhat(H)>\irhat(H')
\stackrel{\text{minimality of }H}{=}
\ihat(H')=\hhat(H')=\hhat(H)\geq\hhat(H-v)\geq\ihat(H-v),
\]
where the inequality $\hhat(H)\geq\hhat(H-v)$ was explained twice earlier.
But $\ihat(H-v)>\ihat(H-v)$ is a contradiction.

\textbf{Case 3: $\delta(H)\geq3$.}
Then
\[
\irhat(H)
=2|V(H)|-|E(H)|
\leq \frac{|V(H)|}{2}
<\left\lceil\frac{|V(H)|}{2}\right\rceil+2=\irhat\left(K_{1,\lceil|V(H)|/2\rceil}\right),
\]
where the star $K_{1,\lceil|V(H)|/2\rceil}$ is a proper collapse of $H$ obtained by mapping all vertices of the smaller part to the core of the star (and if the larger class has more vertices than $\left\lceil\frac{|V(H)|}{2}\right\rceil$, merging some vertices of the larger class).
$H$ is therefore not $R$-critical.

\hspace{1cm}\textbf{Case 3.1: $H$ is not 2-connected.} Since $H$ is connected but not 2-connected, it must have a cut vertex. It is therefore expressible as the wedge sum of two graphs $H_1$ and $H_2$, each with fewer vertices than $H$. By Lemma~\ref{lemma:unions}, if $\ihat(H)\neq \irhat(H)$, then either $\ihat(H_1)\neq\irhat(H_1)$ or $\ihat(H_2)\neq\irhat(H_2)$. Without loss of generality, assume the former.

By the induction hypothesis, $H_1$ is not $P$-critical. This means that there exists a graph $H_1'$ which is a proper collapse of $H_1$ and which satisfies $\ihat(H_1)\leq\ihat(H_1')$. Let $v_1$ and $v_2$ be the vertices at which $H_1$ and $H_2$ are joined to form $H$. Let $H'$ be formed by taking the wedge sum of $H_1'$ at the homomorphic image of $v_1$ in $H'_1$, and of $H_2$ at $v_2$. This graph is a proper collapse of $H$ and, by Lemma~\ref{lemma:unions}, we have $\ihat(H)\leq\ihat(H')$, thus $H$ is not $P$-critical, so $H$ is not a counterexample to our statement.

\hspace{1cm}\textbf{Case 3.2: $H$ is $2$-connected.} In this case, by Lemma~\ref{lemma:game}(\ref{item:deg2}), we have $\ihat(H)\leq \frac{|V(H)|}{2}+1$.
As before for $(R_q)_q$,
\[
\frac{|V(H)|}{2}+1
<\left\lceil\frac{|V(H)|}{2}\right\rceil+2=\ihat\left(K_{1,\lceil|V(H)|/2\rceil}\right),
\]
meaning
that $H$ is not $P$-critical. Hence $H$ cannot be a counterexample to our statement.
\end{proof}

We have now reduced our problem to proving that $\ihat(H)\leq\frac{|V(H)|}{2}+1$ for $2$-connected graphs with $\delta(H)\geq 3$. This is a substantial step down from our main theorem. Unfortunately this statement by itself is not easy to prove by induction, so just like in Proposition~\ref{prop:critical}, we will use induction on a more general statement. In order to state it properly, we will need to define some concepts related to connectivity.

\begin{defi}
Let $H$ be a graph. A \emph{bridge} is an edge $e=uv$ satisfying that in the graph $H-e$, the vertices $u$ and $v$ lie in different components. 
\end{defi}

\begin{defi}
Let $H$ be a graph. A \emph{block} of $H$ is an induced subgraph $B\subseteq H$ which is maximally $2$-connected (that is, it is not a proper superset of another induced $2$-connected subgraph of $H$) and it is not a bridge. This definition differs slightly from the more traditional definition, in which bridges are not excluded.
\end{defi}

This redefinition of block means that several classical results about them need to be rephrased. For example, every edge which is not a bridge is contained in exactly one block. Another result that can be rephrased is the existence of the block-cut tree:

\begin{lemma}[\cite{harary}, Chapter 4]\label{lemma:blockcut}
Let $H$ be a connected graph, let $S_1$, $S_2$ and $S_3$ be the sets of its blocks, bridges and cut vertices, respectively. Construct a graph $G$ by taking $S_1\cup S_2\cup S_3$ as the vertex set, and connecting each cut vertex $v$ to every block and bridge that contains $v$. Then $G$ is a tree.
\end{lemma}

\begin{defi}
Let $H$ be a graph. A block $B$ of $H$ is a \emph{pseudoleaf} if it contains at most one vertex which is adjacent to vertices outside of $B$. If such a vertex exists, we call it the \emph{linking vertex} of $B$.
\end{defi}

\begin{corollary}\label{cor:pseudoleaf}
A graph $H$ with $\delta(H)\geq 2$ and which contains a bridge has at least two pseudoleaves.
\end{corollary}

\begin{proof}
We can assume that $H$ is connected, otherwise restrict to the component of $H$ that contains the bridge.
Since $\delta(H)\geq2$, the graph $H$ cannot be $K_2$, and so the block-cut tree $T$ of $H$ is not a single vertex. Let us consider a leaf $v$ of $T$.
This leaf does not correspond to a cut-vertex of $H$, since a cut-vertex is incident to at least two blocks or bridges (since it separates them). $v$ cannot be a bridge either, since both endpoints of a bridge are cut vertices, again owing to $\delta(H)\geq2$. $v$ must therefore be a block which is incident to at most one cut vertex. In other words, $v$ corresponds to a pseudoleaf. But any tree that is not a single vertex contains at least two leaves, so we are done.
\end{proof}

\begin{lemma}\label{lemma:game}
\begin{enumerate}[(a)]
\item\label{item:deg2} Let $H$ be a $2$-connected graph with $\delta(H)\geq2$, and let $v_2(H)$ be the number of its vertices of degree 2. Then $\ihat(H)\leq \frac{|V(H)|+\max\{v_2(H),2\}}{2}$.
\item\label{item:leaves} Let $H$ be a connected graph with $\delta(H)\geq3$, and let $\ell(H)$ be its number of pseudoleaves. Then $\ihat(H)\leq \frac{|V(H)|+\max\{\ell(H),2\}}{2}$.
\end{enumerate}
\end{lemma}

The statement that we wanted to prove, namely $\ihat(H)\leq\frac{|V(H)|}{2}+1$ for $2$-connected graphs with $\delta(H)\geq 3$, is implied by each of \ref{item:deg2} and \ref{item:leaves}.

Before we start the proof, let us briefly discuss the role of $\max\{v_2(H),2\}$ and $\max\{\ell(H),2\}$ in Lemma~\ref{lemma:game}. For the former, let $H$ be 2-connected. Suppose that $H$ is an induced subgraph of a strictly larger graph $H'$ with $\delta(H')\geq 3$, and that $H$ is not a pseudoleaf in $H'$. What is the least possible number of edges in $H'$ that go from $H$ to the rest of the graph?
Either $H$ is a block of $H'$, in which case there are at least two such edges by the definition of pseudoleaf, or $H$ is not a block of $H'$. In the latter case, $H$ is a single vertex or a bridge, so there are at least three edges from $H$ to $H'\setminus H$ owing to $\delta\left(H'\right)\geq3$, or $H$ is not \emph{maximally} 2-connected, again giving us at least two edges from $H$ to $H\setminus H'$.
Independently, since $\delta(H')\geq 3$, every vertex of degree 2 in $H$ must be incident to an edge connecting it to $H'\setminus H$, so together we get the lower bound of $\max\{v_2(H),2\}$ on the number of edges between $H$ and $H'\setminus H$.

Next let $H$ be a connected graph with $\delta(H)\geq 3$, and suppose that $H$ is an induced subgraph of a strictly larger $2$-connected graph $H'$. What is the least possible number of edges in $H'$ that go from $H$ to the rest of the graph? There must be at least two, since $H'$ is $2$-connected. In addition, if $H$ is not $2$-connected and $S$ is a pseudoleaf of $H$, then there must exist an edge of $H'$ that connects $S$ to $H'\setminus H$ that does not hit the linking vertex of $S$. Since every vertex that is not a linking vertex is contained in at most one pseudoleaf, there are at least $\ell(H)$ edges from $H$ to $H'\setminus H$, giving a lower bound of $\max\{\ell(H),2\}$.

On the other hand, both maxima are necessary: for example, the incidence graph $P$ of the Pappus configuration is a $2$-connected 3-regular graph with $|V(P)|=18$ and $\ihat(P)=10$.

\begin{proof}[Proof of Lemma~\ref{lemma:game}]

Suppose the statement is not true and let $H$ be a minimal counterexample (minimal with respect to number of edges this time) to either (\ref{item:deg2}) or (\ref{item:leaves}). We will do a thorough case analysis to show that $H$ cannot exist. Observe that if $H$ is $2$-connected and $\delta(H)\geq 3$, then both parts of the statement give the same bound, since $v_2(H)=0$ and $\ell(H)=1$.

\textbf{Case 1: $\delta(H)=2$.} We produce a graph sequence $H_0, H_1, H_2, \dots$ as follows: we start with $H_0=H$. Given $H_i$, if it contains at least one vertex with degree at most 2, choose one of them as $v_i$ and set $H_{i+1}=H_i-v_i$. Otherwise we stop the sequence. The last term of the sequence, $H_k$, is either empty or satisfies $\delta(H_k)\geq 3$. Note that since $\delta(H)=2$, we have $k>0$.

Let $W_1, W_2, \dots, W_r$ be the connected components of $H_k$ (where we potentially have $r=0$ if $H_k$ is empty). By Lemma~\ref{lemma:unions} and minimality of $H$, we have $\ihat(H_k)=\sum_{i=1}^r\ihat(W_i)\leq \sum_{i=1}^r\frac{|V(W_i)|+\max\{\ell(W_i),2\}}{2}$. Since $v_i$ has degree at most 2 in $H_i$, by Corollary~\ref{cor:remove} we have $\ihat(H_{i+1})\geq\ihat(H_i)-2+d(v_i)$.
Chaining the $k$ successive inequalities, we obtain
\begin{align*}
\ihat(H_k)&\geq \ihat(H_0)-2k+\sum_{i=0}^{k-1}d(v_i)\\
&=\ihat(H)-2k+|E(H)\setminus E(H_k)|.
\end{align*}
Let us estimate the value of $|E(H)\setminus E(H_k)|$. Consider the graph $W_0$ formed by the vertices $V(H)$ and edges in $E(H)\setminus E(H_k)$. As pointed out before the start of the proof of Lemma~\ref{lemma:game}, the vertices of each $W_i$ are incident to at least $\max\{\ell(W_i),2\}$ edges in $H\setminus W_i$. On the other hand, since all but $v_2(H)$ vertices in $\{v_0, v_1, \dots, v_{k-1}\}$ have degree at least 3 in $H$ (and the remaining vertices have degree 2), the sum of their degrees in $W_0$ is at least $3k-v_2(H)$. Adding the degrees of all vertices, we have that $|E(W_0)|\geq\frac{3k-v_2(H)}{2}+\sum_{i=1}^r\frac{\max\{\ell(W_i),2\}}{2}$. We conclude that \begin{align*}
    \ihat(H)\leq& \ihat(H_k)+2k-|E(H)\setminus E(H_k)|\\
    \leq& \sum_{i=1}^r\frac{|V(W_i)|+\max\{\ell(W_i),2\}}{2}+2k-\frac{3k-v_2(H)}{2}-\sum_{i=1}^r\frac{\max\{\ell(W_i),2\}}{2}\\
    =&\frac{k+v_2(H)}{2}+\sum_{i=1}^r\frac{|V(W_i)|}{2}\leq \frac{|V(H)|+\max\{v_2(H),2\}}{2},
\end{align*}
as we wanted to prove. Hence $H$ is not a counterexample to the lemma.

\textbf{Case 2: $\delta(H)\geq 3$, $H$ is $2$-connected.}
By Lemma \ref{lemma:removeedge} we can remove an edge $e$ so that $H-e$ is still $2$-connected. The number of injections could not decrease as we remove $e$, the number of vertices does not change, and $v_2(H)=0$, $v_2(H-e)\le2$ thus $\max\{v_2(\cdot),2  \}$ does not change. Thus if $H$ is a counterexample, then so is $H-e$, but with fewer edges, contradicting minimality of $H$.

\textbf{Case 3: $\delta(H)\geq 3$, $H$ is not $2$-connected.} In this case $H$ has at least one cut-vertex, meaning that all the blocks of $H$ satisfy part (\ref{item:deg2}). The next lemma in our proof states precisely that in these circumstances, $H$ satisfies (\ref{item:leaves}):

\begin{proposition}\label{prop:finaltree}
Let $H$ be a connected graph with $\delta(H)\geq 3$ which is not $2$-connected. If every block $B$ of $H$ satisfies $\ihat(B)\leq\frac{|V(B)|+\max\left\{v_2(B),2\right\}}{2}$, then $H$ satisfies $\ihat(H)\leq \frac{|V(H)|+\max\{\ell(H),2\}}{2}$.
\end{proposition}

This means that $H$ is not a counterexample to the lemma, completing the proof.
\end{proof}

The reason we state Proposition~\ref{prop:finaltree} as a standalone claim is because the induction step here is incompatible with the one in Lemma~\ref{lemma:game}: on some occasions we will actually want to increase the number of edges (this will happen when we add bridges between blocks). In order to properly define our induction we will need the following proposition:

\begin{proposition}\label{prop:bridges}
There exists a function $f$ such that every connected graph $H$ with $k$ blocks and $\delta(H)\geq 3$ has at most $f(k)$ bridges.
\end{proposition}

\begin{proof}

Consider the block-cut tree $T$ of $H$ (the tree obtained in Lemma~\ref{lemma:blockcut}), and let $\mathcal{S}$ be the set of vertices corresponding to blocks of $H$. 
If $H$ has a bridge, then by Corollary \ref{cor:pseudoleaf},
every leaf of $T$ is a block, and thus $T$ contains at most $k$ leaves.
Since for trees,
\[
v_1(T)+v_2(T)+v_{\geq3}(T)-1=|V(T)|-1=|E(T)|=\frac{1}{2}\sum_{V(T)}d(v)\geq\frac{1}{2}\cdot v_1(T)+v_2(T)+\frac{3}{2}\cdot v_{\geq3}(T),
\]
we obtain that $T$ contains at most $k-2$ vertices of degree at least 3.
Denote the set of these vertices as $\mathcal{V}_3$.

Consider a bridge $uv$, and its corresponding vertex $z$ in $T$.
Because $\delta(H)\geq 3$, each of the vertices $u$ and $v$ (which are themselves cut vertices) is either
incident to at least three bridges, in which case it is in $\mathcal{V}_3$,
or it is contained in a block of $H$.
Hence there are two elements of $\mathcal{S}\cup \mathcal{V}_3$ such that $z$ is the unique bridge on the path between them. Therefore, $H$ contains at most ${|\mathcal{S}\cup \mathcal{V}_3| \choose 2}\leq{2k-2 \choose 2}$ bridges.
\end{proof}

\begin{proof}[Proof of Proposition~\ref{prop:finaltree}]
Suppose that the statement is not true. Consider a counterexample $H$ that minimizes the number of blocks, and among those counterexamples, select one that \emph{maximizes} the number of bridges. This is well-defined by Proposition~\ref{prop:bridges}. We will do a thorough case analysis to show that $H$ cannot exist.

\textbf{Case 1: $H$ contains a vertex $v$ contained in at least two blocks.} We can express $H$ as the wedge sum of two graphs $H_1$ and $H_2$, each containing at least one of the blocks that $v$ is in. Then $\ihat(H)=\ihat(H_1)+\ihat(H_2)-2$, by Lemma~\ref{lemma:unions}. Let $v_1$ and $v_2$ be the vertices of $H_1$ and $H_2$ corresponding to $v$, and let $H'$ be the union of $H_1$ and $H_2$ connected by the edge $v_1v_2$. 

By Lemma~\ref{lemma:unions} we have $\ihat(H')=\ihat(H_1)+\ihat(H_2)-1=\ihat(H)+1$. $H'$ contains the same blocks as $H$, as many pseudoleaves as $H$, one more vertex and one more bridge. In addition we have $\delta(H')\geq 3$, since $v_1$ and $v_2$ have degree at least 2 in their respective graphs because they lie in a block. Thus $\ihat(H')\leq \frac{|V(H')|+\max\{\ell(H'),2\}}{2}$, and hence
\[
\ihat(H)=\ihat(H')-1\leq \frac{|V(H)|+1+\max\{\ell(H),2\}}{2}-1<\frac{|V(H)|+\max\{\ell(H),2\}}{2}.
\]
So $H$ cannot be a counterexample.

\textbf{Case 2: $H$ contains a vertex $v$ not contained in any block.} In this case, every edge incident to $v$ is a bridge. Let $W_1, W_2, \dots, W_k$ be the components of $H-v$, and observe that $k=d(v)\geq 3$. Applying Lemma~\ref{lemma:unions} repeatedly, we obtain that $\ihat(H)=-(k-2)+\sum_{i=1}^k\ihat(W_i)$. 

For any $i,j\in [k]$, let $W_{i,j}$ be the union of $W_i$ and $W_j$ with an edge between the corresponding neighbours of $v$. Every pseudoleaf of $H$ is contained in $W_i$ for some $i\in[k]$, and it is a pseudoleaf of $W_{i,j}$ for all $j\neq i$, which means that $\ell(H)=\sum_{i=1}^k\ell\left(W_{i,i+1}\right)/2$, where the subindex $k+1$ is identified with 1. In addition, by Lemma~\ref{lemma:unions}, and using Corollary~\ref{cor:pseudoleaf}, we have
$\ihat(W_i)+\ihat(W_j)-1=\ihat(W_{i,j})\leq \frac{|V(W_i)|+|V(W_j)|+\ell\left(W_{i,j}\right)}{2}$,
so
\begin{align*}
\ihat(H)=& \sum_{i=1}^k\ihat(W_i)-k+2= \sum_{i=1}^k\frac{\ihat(W_i)+\ihat(W_{i+1})}{2}-k+2=\sum_{i=1}^k\frac{\ihat(W_i)+\ihat(W_{i+1})-1}{2}-\frac{k-4}{2}\\ \leq& \sum_{i=1}^k\frac{|V(W_i)|+|V(W_{i+1})|+\ell\left(W_{i,i+1}\right)}{4}-\frac{k-4}{2}=\frac{|V(H)|-1+\ell(H)}{2}-\frac{k-4}{2}\\ \leq&\frac{|V(H)|+\ell(H)}{2},
\end{align*}
since $k\geq 3$. Thus $H$ cannot be a counterexample.

\textbf{Case 3: Every vertex in $H$ is contained in exactly one block.} Let $H'$ be the graph obtained from $H$ by contracting every block of $H$. Because every vertex is contained in exactly one block of $H$, the vertices of $H'$ are in bijection with the blocks of $H$. $H'$ is connected and acyclic, so it is a tree.

Let $\mathcal{B}$ be the set of blocks of $H$ and let $B\in\mathcal{B}$. Let $v_B$ be the vertex of $H'$ corresponding to $B$. Notice that $\delta(B)\geq 2$, and any vertex in $B$ with degree 2 must be incident to a bridge in $H$, so $d(v_B)\geq v_2(B)$. If, in addition,$B$ is not a pseudoleaf, then it must be incident to at least two bridges, so in this case $d(v_B)\geq\max\{v_2(B),2\}$. On the other hand, if $B$ is a pseudoleaf, then there is exactly one vertex of $B$ which is incident to a bridge (and thus might have degree 2 in $B$). This means that $d(v_B)\geq\max\{v_2(B),2\}-1$. Altogether, that makes \[2|\mathcal{B}|-2=2e(H')=\sum_{B\in\mathcal{B}}d(v_B)\geq -\ell(H)+\sum_{B\in\mathcal{B}}\max\{v_2(B),2\}.\]

Applying Lemma~\ref{lemma:unions} repeatedly, we see that $\ihat(H)=\sum_{B\in\mathcal{B}}\ihat(B)-|\mathcal{B}|+1$, as every bridge decreases $\ihat$ by one. Since each block satisfies $\ihat(B)\leq\frac{|V(B)|+\max\{v_2(B),2\}}{2}$, we can put this together:

\[\ihat(H)\leq\sum_{B\in\mathcal{B}}\frac{|V(B)|+\max\left\{v_2(B),2\right\}}{2}-|\mathcal{B}|+1\leq \frac{|V(H)|+\ell(H)}{2}.\]

Thus $H$ cannot be a counterexample, and the proof is complete.

\end{proof}

\section{Discussion and open problems}

In our proof of log-convergence of $(G_q)_{q\in Q}$, we did not need to resort to heavily exploiting the algebraic properties of the underlying field $\mathbb{F}_q$. Crucially, we used that $\text{Aut}(G_q)$ acts 2-transitively on $V_1(G_q)$ to prove Lemma~\ref{lemma:unions}. This phenomenon of 2-transitivity holds exactly when a projective plane is $PG(2,K)$, for $K$ a division ring, but if one could establish Lemma~\ref{lemma:unions} using general arguments of projective geometry, a path of proving log-convergence of general sequences of the incidence graphs of finite projective planes would be open.

\begin{question}\label{quest:generalplanes}
    Let $(G_n)_{n=1}^\infty$ be a sequence of incidence graphs of finite projective planes which are not necessarily field planes. Does $(G_n)_{n=1}^\infty$ always log-converge?
\end{question}

In the classical notion of graph convergence, the Chung-Graham-Wilson theorem establishes the equivalence of several conditions on graph sequences, which are collectively referred to as ``quasirandomness". One of those conditions is convergence to the constant graphon, i.e., having the same limit as $G(n,p)$, where $p\in[0,1]$ is a fixed parameter. We would be interested in finding analogues of some of these conditions. In particular, we would be interested in two types of conditions as follows.
\begin{question}
Does there exist a finite family $\mathcal{F}\subseteq\mathcal{B}_0$, such that for every sequence $(G_n)_{n=1}^\infty$ of graphs in $\mathcal{B}_0$, if
\[
\lim\limits_{n\rightarrow\infty}\frac{\log \Hom(H,G_n)}{-\log t_B(K_2, G_n)}=\hrhat(H)
\]
holds for every $H\in \mathcal{F}$, then it also holds for every $H\in \mathcal{B}_0$?
\end{question}

An implication of the Chung-Graham-Wilson theorem \cite{ChungGrahamWilson89} is that the densities $t(K_2, G_n)$ and $t(C_4, G_n)$ are enough to determine quasirandomness. In particular, if $\ler{t(K_2, G_n)}_{n=1}^\infty$ converges and $\lim\limits_{n\rightarrow\infty}\frac{-\log t(C_4, G_n)}{-\log t(K_2, G_n)}=4$, then $\ler{G_n}_{n=1}^\infty$ is quasirandom. Szegedy asked whether a similar limit, with $t_B$ replacing $t$, is enough to guarantee that the sequence $\ler{G_n}_{n=1}^\infty$ converges to a linear combination of certain random graphs. 

If this was the case, we could answer Question~\ref{quest:generalplanes} affirmatively.
This is because, if $G$ is the incidence graph of a projective plane of order $n$, then $\Hom(C_4, G)=\Theta(n^4)$, with the image of each homomorphism being either a single edge or a cherry.

Another, more restrictive condition, involves the spectrum of the adjacency matrix of the graphs. The eigenvalues of the incidence graph of a projective plane depend only on its order $q$, and can be computed explicitly. If the log-convergence of a sequence of graphs to the quasirandom limit depends only on its spectrum, as is the case in the classical notion of convergence, then that would also imply the log-convergence of all incidence graphs of projective planes.

\begin{question}
Given an infinite sequence $\ler{H_n}_{n=1}^\infty$ of graphs in $\mathcal{B}_0$, is it possible to determine whether it converges to the same limit as $(R_n)_{n=1}^\infty$ just by knowing the spectrum of the graphs $H_n$?
\end{question}

In an opposite direction to Question~\ref{quest:generalplanes}, we return to considering only the standard projective planes $PG(2,p)$, for $p$ a prime number. As we saw in Section~\ref{section:preliminaries}, the sequence $\ler{\log_q\inj(G_2, G_q)}_{q\in Q}$ does not converge, as its terms are $-\infty$ when $q$ is an odd prime power and non-negative when $q$ is a power of 2. Can something similar still happen if $q$ is restricted to only take prime values? Remember that this restriction was included in Szegedy's original question about log-convergence of projective planes, even though it turned out to be unnecessary. Could this restriction be enough to guarantee convergence of the number of injective homomorphisms?

\begin{question}
Does $\ler{\log_p\inj(H, G_p)}_{p\in P}$ converge for all $H\in \mathcal{B}_0$?
\end{question}

Finally, the structurally perhaps most interesting question is to find a compact representation of the limit object of log-convergence, akin to graphons and permutons.  This would be the most useful piece of the theory to put in place, even if we do not phrase it as a numbered question, for it would be inherently informal.

\bibliography{libr}

@article{Szegedy2015SparseGL,
  title={Sparse graph limits, entropy maximization and transitive graphs},
  author={Bal{\'a}zs Szegedy},
  journal={arXiv: Combinatorics 1504.00858},
  year={2015}
}

@article{Sidorenko,
  author = {Sidorenko, Alexander},
  year = {1993},
  month = {06},
  pages = {201-204},
  title = {A correlation inequality for bipartite graphs},
  volume = {9},
  journal = {Graphs and Combinatorics},
  doi = {10.1007/BF02988307}
}

@article{chartrand1972critically,
  title={Critically $n$-connected graphs},
  author={Chartrand, Gary and Kaugars, Agnis and Lick, Don R},
  journal={Proceedings of the American Mathematical Society},
  volume={32},
  number={1},
  pages={63--68},
  year={1972}
}

@book{harary,
  author = {Harary, Frank},
  year = {1969},
  title = {Graph Theory},
  publisher = {Addison-Wesley Publishing Company}
}

@article{ChungGrahamWilson89,
  author  = {Chung, Fan R. K. and Graham, Ronald L. and Wilson, Richard M.},
  title   = {Quasi-random graphs},
  journal = {Combinatorica},
  volume  = {9},
  number  = {4},
  pages   = {345--362},
  year    = {1989}
}

@article{AldousExchangeable,
  author = {Aldous, David J.},
  title = {Representations for partially exchangeable arrays of random variables},
  journal = {Journal of Multivariate Analysis},
  volume = {11},
  number = {4},
  pages = {581--598},
  year = {1981}
}

@article{exchangeableVSgraphlimits,
  author = {Diaconis, Persi and Janson, Svante},
  title = {Graph limits and exchangeable random graphs},
  journal = {Rendiconti di Matematica e delle sue Applicazioni. Serie VII},
  volume = {28},
  pages = {33--61},
  year = {2008},
  archivePrefix = {arXiv},
  eprint = {0712.2749}
}

@article{LovSzegedy2006,
  author = {Lov{\'a}sz, L{\'a}szl{\'o} and Szegedy, Bal{\'a}zs},
  title = {Limits of dense graph sequences},
  journal = {Journal of Combinatorial Theory, Series B},
  volume = {96},
  number = {6},
  pages = {933--957},
  year = {2006}
}

@article{BorgsChayesLovSosVesztergombiI,
  author = {Borgs, Christian and Chayes, Jennifer T. and Lov{\'a}sz, L{\'a}szl{\'o} and S{\'o}s, Vera T. and Vesztergombi, Katalin},
  title = {Convergent graph sequences {I}: Subgraph frequencies, metric properties and testing},
  journal = {Advances in Mathematics},
  volume = {219},
  number = {6},
  pages = {1801--1851},
  year = {2008}
}

@article{LeftRightBounded,
author = {Borgs, Christian and Chayes, Jennifer and Kahn, Jeff and Lovász, László},
title = {Left and right convergence of graphs with bounded degree},
journal = {Random Structures \& Algorithms},
volume = {42},
number = {1},
pages = {1-28},
doi = {https://doi.org/10.1002/rsa.20414},
url = {https://onlinelibrary.wiley.com/doi/abs/10.1002/rsa.20414},
eprint = {https://onlinelibrary.wiley.com/doi/pdf/10.1002/rsa.20414},
year = {2013}
}

@article{HatLovSzegLocalGlobal,
  author = {Hatami, Hamed and Lov{\'a}sz, L{\'a}szl{\'o} and Szegedy, Bal{\'a}zs},
  title = {Limits of locally--globally convergent graph sequences},
  journal = {Geometric and Functional Analysis},
  volume = {24},
  pages = {269--296},
  year = {2014},
}

@article{BollRiordanSparse,
author = {Bollobás, Béla and Riordan, Oliver},
title = {Sparse graphs: Metrics and random models},
journal = {Random Structures \& Algorithms},
volume = {39},
number = {1},
pages = {1-38},
url = {https://onlinelibrary.wiley.com/doi/abs/10.1002/rsa.20334},
eprint = {https://onlinelibrary.wiley.com/doi/pdf/10.1002/rsa.20334},
year = {2011}
}

@article{LpGraphonsI,
  author  = {Borgs, Christian and Chayes, Jennifer T. and Cohn, Henry and Zhao, Yufei},
  title   = {An {$L^p$} theory of sparse graph convergence {I}: Limits, sparse random graph models, and power law distributions},
  journal = {Transactions of the American Mathematical Society},
  volume  = {372},
  number  = {5},
  pages   = {3019--3062},
  year    = {2019}
}

@article{GraphonProcesses,
  author = {Borgs, Christian and Chayes, Jennifer T. and Cohn, Henry and Holden, Nina},
  title = {Sparse Exchangeable Graphs and Their Limits
via Graphon Processes},
  journal = {Journal of Machine Learning Research},
  volume = {18},
  pages = {1--71},
  year = {2018}
}

@article{ActionConvergence,
  author = {Backhausz, {\'A}gnes and Szegedy, Bal{\'a}zs},
  title = {Action convergence of operators and graphs},
  journal = {Canadian Journal of Mathematics},
  volume = {74},
  number = {1},
  pages = {72--121},
  year = {2022}
}

@article{permutons,
  author  = {Hoppen, Carlos and Kohayakawa, Yoshiharu and Moreira, Carlos Gustavo
             and R{\'a}th, Bal{\'a}zs and Sampaio, Rudini Menezes},
  title   = {Limits of permutation sequences},
  journal = {Journal of Combinatorial Theory, Series B},
  volume  = {103},
  number  = {1},
  pages   = {93--113},
  year    = {2013}
}

@article{BenjaminiSchramm2001,
	author = {Itai Benjamini and Oded Schramm},
	title = {Recurrence of Distributional Limits of Finite Planar Graphs},
	journal = {Electronic Journal of Probability},
	volume = {6},
	year = {2001},
	pages = { Paper no. 23, 1-13},
	issn = {1083-6489},
    }

@article{ElekSzegedy2012,
  author  = {Elek, G{\'a}bor and Szegedy, Bal{\'a}zs},
  title   = {A measure-theoretic approach to the theory of dense hypergraphs},
  journal = {Advances in Mathematics},
  volume  = {231},
  number  = {3--4},
  pages   = {1731--1772},
  year    = {2012}
}

@article{matroids,
  author  = {B{\'e}rczi, Krist{\'o}f and Borb{\'e}nyi, M{\'a}rton and Lov{\'a}sz, L{\'a}szl{\'o} and T{\'o}th, L{\'a}szl{\'o} M{\'a}rton},
  title   = {Quotient-Convergence of Submodular Setfunctions},
  journal = {Combinatorica},
  volume  = {46},
  eid     = {6},
  year    = {2026}
}

@article{posetsJanson,
  author  = {Janson, Svante},
  title   = {Poset limits and exchangeable random posets},
  journal = {Combinatorica},
  volume  = {31},
  pages   = {529--563},
  year    = {2011}
}
\bibliographystyle{plain}

\end{document}